\documentclass{article}

\usepackage{amsmath}
\usepackage{amsthm}
\usepackage{amsxtra}
\usepackage{graphicx}
\usepackage{epsfig}
\usepackage{amsfonts,amssymb}
\usepackage{latexsym}
\usepackage{enumitem}
\usepackage{inputenc}
\newtheorem{thm}{Theorem}[section]

\newtheorem{lem}[thm]{Lemma}
\newtheorem{prop}[thm]{Proposition}

\newtheorem{Quest}{Question}

\theoremstyle{definition}

\numberwithin{equation}{section}

\newcommand{\supp}{\mathrm{supp}}

\newcommand{\dd}{\mathrm{d}}

\newcommand{\spec}{\mathrm{sp}}

\newcommand{\mi}{\mathrm{I}}
\renewcommand{\u}{\beta}
\renewcommand{\v}{\gamma}
\newcommand{\w}{\delta}

\newcommand{\eg}{\text{e.g.\,}}
\newcommand{\ie}{\text{i.e.\;\,}}

\begin{document}
\title{On the question "Can one hear the shape of a group?" and Hulanicki type theorem for graphs}
\author{ {\bf Artem Dudko
\thanks{A. Dudko acknowledges the support by the National Science Centre, Poland,
grant 2016/23/P/ST1/04088 under POLONEZ programme which has received funding from the EU\;\protect\includegraphics[width=.03\linewidth]{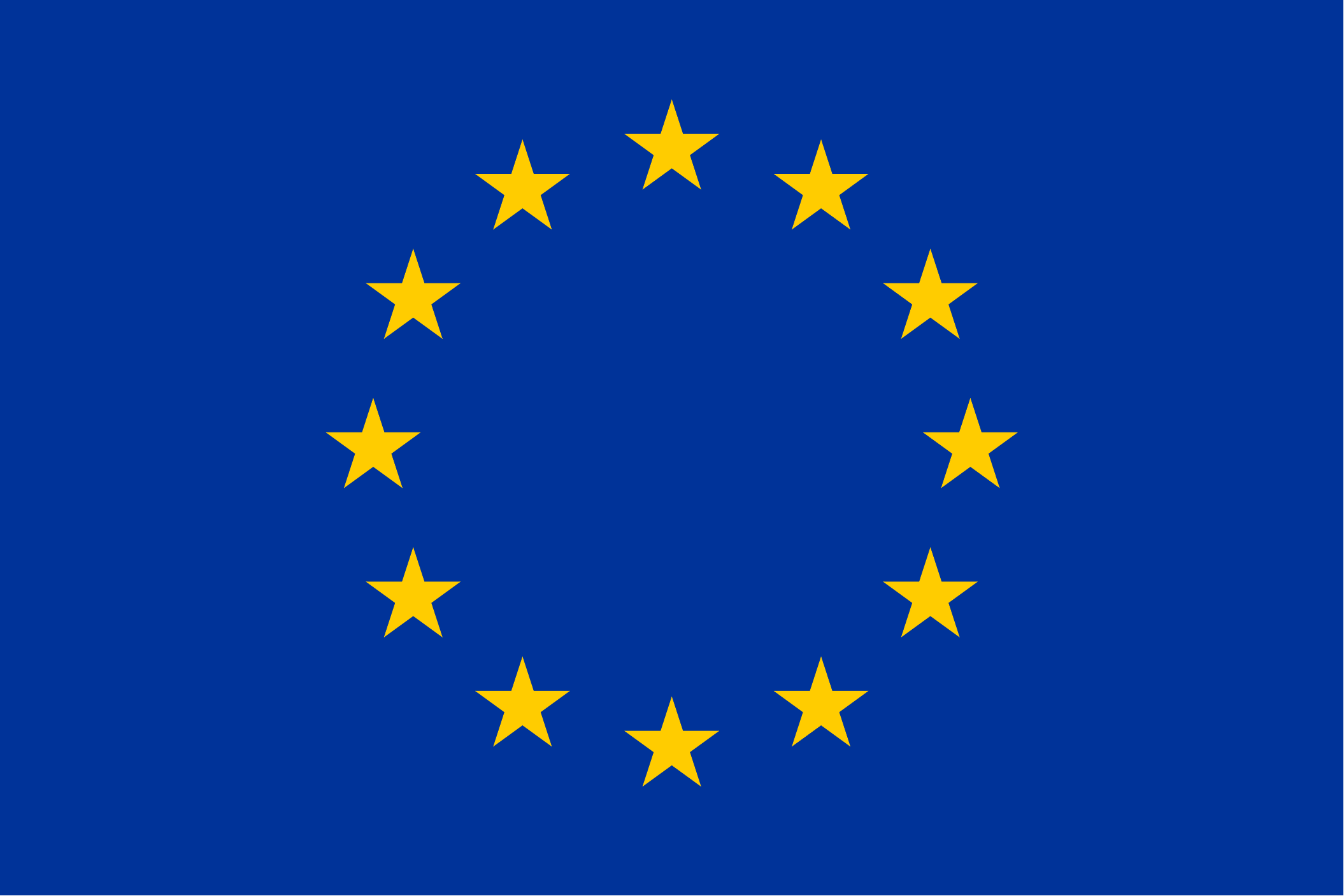} Horizon 2020 research and innovation programme under the MSCA grant agreement No. 665778}}
\\
                    IM PAN, Warsaw, Poland  \\
          adudko@impan.pl \\
         {\bf Rostislav Grigorchuk
          \thanks{Rostilav Grigorchuk was partially supported by NSF grant DMS-1207699 and by Simons Foundation  Collaboration Grant for  Mathematicians,  Award Number  527814}}\\        Texas A\&M University, College Station, TX, USA  \\      grigorch@math.tamu.edu }

\date{}

\maketitle
\begin{abstract} We study the question of whether or not it is possible to determine a finitely generated group $G$ up to some notion of equivalence from the spectrum $\spec(G)$ of $G$. We show that the answer is "No" in a strong sense. As a first example we present the collection of amenable 4-generated groups $G_\omega$, $\omega\in\{0,1,2\}^\mathbb N$, constructed by the second author in 1984. We show that among them there is a continuum of pairwise non-quasi-isometric groups with  $\spec(G_\omega)=[-\tfrac{1}{2},0]\cup[\tfrac{1}{2},1]$. Moreover, for each of these groups $G_\omega$ there is a continuum of covering groups $G$ with the same spectrum. As a second example we construct a continuum of $2$-generated torsion-free step-3 solvable groups with the spectrum $[-1,1]$. In addition, in relation to the above results, we prove a version of  the Hulanicki Theorem about inclusion of spectra for covering graphs.
\end{abstract}
\section{Introduction}
The first part of the title is related to the famous paper of Mark Kac \cite{Kac-Shape-66}: "Can one hear the shape of a drum?". The above question can be traced back to Lipman Bers and Hermann Weyl. It concerns a plane domain $\Omega$ with piecewise smooth boundary (a \emph{drum}) and the Laplace operator
$$\Delta=-\left(\frac{\partial ^2}{\partial x^2}+\frac{\partial ^2}{\partial y^2}\right).$$ The question is whether or not the domain $\Omega$ can be determined (up to isometry) from the spectrum of $\Delta$. In \cite{GordonWebbWolpert-Isospectral-92} the authors gave a negative answer by constructing a pair of regions in the plane of different shapes but with identical eigenvalues of the Laplace operator.

In this paper we consider the analogous question for the Laplace operator on a Cayley graph of an infinite finitely generated group. Given a finitely generated group $G$ with a symmetric generating set $S$, the Laplace operator of the Cayley graph $\Gamma=\Gamma(G,S)$ of $G$ acts on $l^2(G)$ by
\begin{equation}\label{EqLapCayley}(\Delta f)(g)=|S|-\sum\limits_{s\in S}f(s^{-1}g).
\end{equation} The Laplace operator is strongly related to the Markov operator $$M=1-\tfrac{1}{|S|}\Delta,\;\;(Mf)(g)=\tfrac{1}{|S|}\sum\limits_{s\in S}f(s^{-1}g),$$ corresponding to the simple random walk on $G$. The spectrum of $M$ is called the spectrum of the group $G$ and is denoted by $\spec(G)$. A natural question, inspired by the classical  question about the shape of a drum, is whether or not it is possible to determine the group $G$, up to some notion of similarity, from the spectrum of $G$. This question is discussed in \cite{Valette-ShapeGroup-96} (see also \cite{HarpeRobertsonValette-SpectrumGroupI-93}, \cite{HarpeRobertsonValette-SpectrumGroupII-93} and \cite{Fujiwara-ShapeGroup-16}). As the author of \cite{Valette-ShapeGroup-96} points it out, it is easy to see that it is not possible to determine $G$ up to isomorphism just from the spectrum of $G$. For example, for $\mathbb Z^n$, $n\in\mathbb N$, with $S=\{e_1,\ldots,e_n,-e_1,\ldots,-e_n\}$, where $\{e_i\}_{i=1,\ldots,n}$ is the standard basis of $\mathbb Z^n$, the spectrum is equal to $[-1,1]$ and so does not depend on $n$. Other examples are non-isomorphic groups with isomorphic Cayley graphs. Such examples were constructed for instance in \cite{BozejkoDykemaLehner-Cayley-06}. Notice also that even a finite group can have non-isomorphic Cayley graphs (corresponding to different generating sets) with the same spectrum (see \eg \cite{AbdollahiJanbazGhahramani-CospectralCayley-17}, \cite{Babai-SpectraCayley-79} and \cite{Lubotzky-IsospectralGraphs-06}).

Nevertheless, the spectrum of a group $G$ (in fact, even the spectral radius of the Markov operator $r(M)=\max |\{z:z\in\sigma(M)\}|$) can give a valuable information about its structure. In \cite{Kesten-BanachMean-59,Kesten-RandomWalk-59} Kesten showed that for a group generated by a  symmetric finite set $S$ one has
$$\frac{\sqrt{2n-1}}{n}\leqslant r(M)\leqslant 1,\;\;\text{where}\;\;|S|=2n.$$ Moreover, $r(M)=1$ if and only if $G$ is amenable and $r(M)=\frac{\sqrt{2n-1}}{n}$ for $n\geqslant 2$ if and only if $G$ is free on $S_0$ with $S=S_0\sqcup S_0^{-1}$.

However, many questions about spectra of groups remain open. One of them is: what can be the shape of the spectrum of a group? It is known that the spectrum of a finitely generated group $G$ can be an interval or a union of two intervals (as shown in \cite{DudkoGrigorchuk-Spectrum-17}). But it is not known whether it can be a union of $n\geqslant 3$ disjoint intervals, a countable set of points accumulating to a point, or a Cantor set. We notice that in \cite{Kuhn-Anisotropic-92} Gabriella Kuhn constructed for every $n\geqslant 3$ a non-amenable group $G_n$ generated by a finite set $S_n$ and a Markov operator corresponding to some non-uniformly distributed probabilities on $S_n$ with the spectrum equal to a disjoint union of $n$ intervals, but it remains an open question whether the same can be achieved with all probabilities equal to $\frac{1}{|S_n|}$.

Other natural questions are:
 \begin{itemize}\item{} given a certain closed subset $\Sigma$  of $[-1,1]$ how many finitely generated groups have the spectrum equal to $\Sigma$?
\item{} how does the spectrum change under group coverings?
\end{itemize} In this paper we address certain aspects of the latter two questions.

 Our starting point is a construction of a continuum of groups generated by four involutions with spectra equal to the same union of two intervals.
 These groups are from the family of groups $G_\omega,\omega\in\Omega=\{0,1,2\}^\infty$, introduced by the second author in \cite{Grigorchuk-DegreesGrowth-84}. The most known of these groups is the group $\mathcal G=G_{\omega_0}$, where $\omega_0$ is the periodic sequence $012012\ldots$ The groups $G_\omega$ are amenable groups generated by four involutions $a,b_\omega,c_\omega,d_\omega$ which we introduce in Section \ref{SecGomega}. One of the results of \cite{Grigorchuk-DegreesGrowth-84} is that the collection of groups $G_\omega$, $\omega\in\Omega$, contains a continuum of groups with pairwise nonequivalent in the Schwarz-Milnor sense growth functions (and thus with pairwise non-quasi-isometric Cayley graphs). Let $\Omega_2$ be the subset of $\Omega$ consisting of sequences with at lease two symbols from $\{0,1,2\}$ occurring infinitely many times. We show:
\begin{thm}\label{ThSpecGrig}  For every $\omega\in\Omega_2$ one has $\spec(G_\omega)=[-\tfrac{1}{2},0]\cup [\tfrac{1}{2},1]$.
\end{thm}
\noindent The above theorem shows that the answer to the question in the title is "No" even when instead of isomorphism one considers a very weak notion of equivalence of groups: quasi-isometry. We notice that $\spec(G_\omega)=[-\tfrac{1}{2},0]\cup [\tfrac{1}{2},1]$ for every $\omega\in \Omega$, but for us it is convenient to consider only $\omega\in \Omega_2$.

In \cite{DudkoGrigorchuk-Spectrum-17}, Theorem 2, the authors already showed that the spectrum of the group $\mathcal G=G_{\omega_0}$, where $\omega_0=(012)^\infty\in\Omega_2$, is  equal to $[-\tfrac{1}{2},0]\cup [\tfrac{1}{2},1].$
\noindent Theorem \ref{ThSpecGrig} can be proven similarly to Theorem 2 from \cite{DudkoGrigorchuk-Spectrum-17}. In \cite{GrigorchukPerezNagnibeda-Schreier-18} Grigorchuk, P\'erez and Smirnova-Nagnibeda, using different methods, found the spectra of the so called spinal groups which include the groups $G_\omega,\omega\in \Omega$, and the result given in Theorem \ref{ThSpecGrig} interferes with Theorem 1.2 from \cite{GrigorchukPerezNagnibeda-Schreier-18}. In this paper we obtain Theorem \ref{ThSpecGrig} as a corollary of the following:
\begin{thm}\label{PropGrigSpec} Let $G$ be an amenable group generated by four involutions $\tilde a,\tilde b,\tilde c,\tilde d$ such that $\tilde b\tilde c\tilde d=1$. Assume that for some $\omega\in\Omega_2$ there exists a surjection $\varphi:G\to G_\omega$ such that $\varphi(\tilde a)=a,\varphi(\tilde b)=b_\omega,\varphi(\tilde c)=c_\omega$ and $\varphi(\tilde d)=d_\omega$. Then $\spec(G)=[-\tfrac{1}{2},0]\cup[\tfrac{1}{2},1]$.
\end{thm}
\noindent As the application of Theorem \ref{PropGrigSpec} we obtain:
\begin{thm}\label{ThMain} For every $\omega\in\Omega_2$ there exists a continuum of amenable groups covering $G_\omega$ and with the spectrum equal to $[-\tfrac{1}{2},0]\cup[\tfrac{1}{2},1]$. Each of these groups is generated by four involutions $a,b,c,d$ satisfying the condition $bcd=1$.
\end{thm}
\noindent
Notice that by \cite{BenliGrigorchukHarpe-AmenableGroups-13} all of the coverings of $G_\omega$  satisfying the conditions of Theorem \ref{ThMain} (in particular, the group $G_\omega$) are not finitely presented.

In addition, using different techniques based on the classical results of P. Hall \cite{Hall-Finiteness-54} and Higson-Kasparov result on correctness of Baum-Connes Conjecture for groups with the Haagerup property (also known as a-T-menable groups) \cite{Higson-Kasparov-OperatorKTheory-97} we prove the following:
\begin{thm}\label{ThSpec[-1,1]} There is a continuum of pairwise non-isomorphic $2$-generated torsion free step-3 solvable groups with the spectrum $[-1,1]$.
\end{thm}

In the proof of our results we use the Hulanicki Theorem from \cite{Hulanicki-WeakContainment-64}-\cite{Hulanicki-Means-66} (referred sometimes as Hulanicki-Reiter Theorem \cite{BekkaHarpeValette-Kazdan-08})
\begin{thm}[Hulanicki]\label{ThHulanicki} A locally compact group $G$ is amenable if and only if any unitary representation of $G$ is weakly contained in the regular representation of $G$.
\end{thm}
\noindent In particular, Theorem \ref{ThHulanicki} implies that given a subgroup $H<G$ of a countable amenable group $G$  the quasi-regular representation $\lambda_{G/H}$ of $G$  is weakly contained in the regular representation $\lambda_G$. It is known that weak containment $\rho\prec\eta$ of two unitary representations of a group $G$ is equivalent to the inclusion of the spectra $\sigma(\rho(m))\subset\sigma(\eta(m))$ for any $m\in\mathbb C[G]$ (see \cite{BekkaHarpeValette-Kazdan-08} or \cite{Dixmier-Algebras-69}). Notice that given a generating set $S$ for $G$ the Cayley graph $\Gamma(G,S)$ covers the Schreier graph $\Gamma(G,H,S)$ of the action of $G$ on $G/H$ and from Hulanicki Theorem \ref{ThHulanicki} it follows that for an amenable group $G$ the spectrum of $\Gamma(G,S)$ includes the spectrum of $\Gamma(G,H,S)$.

The above observation motivates us to consider a more general situation. Namely, given a graph $\Gamma$ of uniformly bounded degree together with a collection of weights $\alpha$ on the edges of $\Gamma$ in Section \ref{SecGraphs} we introduce a Laplace type operator on $l^2(\Gamma)$ associated to $\alpha$. If a graph $\widetilde \Gamma$ covers $\Gamma$ then the collection of weights $\alpha$ can be lifted to $\widetilde\Gamma$ and one can consider the corresponding Laplace type operator $\widetilde H$ on $l^2(\widetilde\Gamma)$. A natural question inspired by Theorem \ref{ThHulanicki} is: assuming amenability of $\widetilde\Gamma$ is it always true that $\sigma(H)\subset\sigma(\widetilde H)$? We don't know the answer in the full generality. But we prove that the answer is "Yes" under some restrictions:
\begin{thm}[Weak Hulanicki Theorem for graphs]\label{ThGraphCov}  Let $\Gamma_1$ be a uniformly bounded connected weighted graph which covers a weighted graph $\Gamma_2$ such that either
\setlist{nolistsep}
\begin{itemize}\item[$a)$] $\Gamma_1$ is amenable and $\Gamma_2$ is finite or
\item[$b)$] $\Gamma_1$ has subexponential growth.
 \end{itemize} Let $H_1,H_2$ be the Laplace type operators associated to $\Gamma_1,\Gamma_2$. Then $\sigma(H_2)\subset \sigma(H_1)$.
\end{thm}
\noindent This result is used in the proof of Theorem \ref{PropGrigSpec}.

\section{Groups $G_\omega$ and their Schreier graphs}\label{SecGomega}
Let us recall the construction of the family of groups $G_\omega$ introduced by the second author in \cite{Grigorchuk-DegreesGrowth-84}. Notice that, originally, the groups $G_\omega$ were introduced as groups acting on the interval $[0,1]$. We will view them as groups of automorphisms of the binary rooted tree. Recall that the binary rooted tree has the vertex set $V$ identified with the set of all finite words over the two-letter alphabet $\{0,1\}$. Two vertices of $V$ are connected by an edge if and only if one of them is obtained by concatenation of a letter at the end of the other: $v=w0$ or $v=w1$. The $n$th level $V_n$ of $V$ consists of all $n$-letter words: $V_n=\{0,1\}^n$. Thus, each vertex from $V_n$ is connected to one vertex of $V_{n-1}$ for $n\geqslant 1$ and two vertices of $V_{n+1}$ for $n\geqslant 0$.  The set of vertices $V_n$ is equipped with the natural lexicographic order.

As before, we denote by $\Omega=\{0,1,2\}^\mathbb N$ the space of all infinite sequences of elements from $\{0,1,2\}$. We equip $\Omega$ with the Tychonoff topology. Using the correspondence $i\to \overline i,\;i\in\{0,1,2\}$, given by
\begin{equation}\label{EqCorrespondence}
\overline{0}=\begin{bmatrix}\Pi\\ \Pi\\ \mi \end{bmatrix},\;\;\overline{1}=\begin{bmatrix}\Pi\\ \mi \\ \Pi\end{bmatrix},\;\;\overline{2}=\begin{bmatrix}\mi \\ \Pi\\ \Pi\end{bmatrix},
\end{equation} assign to every $\omega\in \Omega$ the sequence of
columns $\overline\omega=\overline\omega_1\cdots\overline\omega_n\cdots$
which we view as a $3\times\infty$ matrix with the entries in the alphabet
$\{\Pi,\mi\}$.
Denote by $\u_\omega,\v_\omega$ and $\w_\omega$ the increasing sequence of indexes $n$ for which the first, the second and the third row of $\overline\omega_n$ correspondingly is equal to $\Pi$.
For $n\in\mathbb N$, let $v_n=1^n$ be the vertex of the $n$th level of $T$ the largest in the lexicographic order. For $n\in\mathbb N$, denote by $w_n$ the vertex from $n$th level of $T$ attached to $v_{n-1}$ and not equal to $v_n$, \ie $w_n=1^{n-1}0$. Denote by $w_0=v_0$ the vertex of $T$ corresponding to the empty word (called \text{the root} of $T$). Let $\sigma_n$ be the transposition of the two branches of $T$ adjacent to $w_n$. Thus, $$\sigma_n(w_n0v)=w_n1v,\;\;\sigma_n(w_n1v)=w_n0v,\;\;\sigma_n(u)=u\;\;\text{for all other vertices},$$ where $v$ is any finite word over $\{0,1\}$. We define the automorphisms $a,b_\omega,c_\omega,g_\omega$ of $T$ by:
\begin{equation}\label{EqGenerators}
a=\sigma_0,\;\;b_\omega=\sigma_{\u_1}\cdots\sigma_{\u_n}\cdots,\;\;
c_\omega=\sigma_{\v_1}\cdots\sigma_{\v_n}\cdots,\;\;
d_\omega=\sigma_{\w_1}\cdots\sigma_{\w_n}\cdots.
\end{equation} The group $G_\omega$ is the group of automorphisms of $T$ generated by $a,b_\omega,c_\omega,d_\omega$.

Recall that $\mathcal G=G_{\omega_0}$, where $\omega_0$ is the periodic sequence $012012\ldots$. In \cite{Lysionok-Grigorchuk-85} Lysionok showed that it has the following presentation:
\begin{equation}\label{EqGrigPres}\mathcal G=\langle a,b,c,d:a^2,b^2,c^2,d^2,bcd,\sigma^k(ad)^4,
\sigma^k(adacac)^4 \;(k\geqslant 0)\rangle, \end{equation} where $\sigma$ is the substitution defined by $\sigma(a)=aca,\sigma(b)=d,\sigma(c)=b,\sigma(d)=c$. Thus, the canonical generators $a,b,c,d$ are involutions (elements of order 2). Moreover, from the relation $bcd=1$ one can deduce that $b,c,d$ commute and $bc=d,bd=c,cd=b$. Therefore, $\mathcal G$, in fact, is generated by three elements: $a$ and any two from $b,c,d$.

More generally, let us recall similar facts about $G_\omega$. We refer the reader to \cite{Grigorchuk-DegreesGrowth-84} for details.
Denote the identity transformation of $T$ by $1$. The elements $b_\omega,c_\omega,d_\omega$ are pairwise commuting. One has:
\begin{equation}\label{EqGomegaProperties}a^2=b_\omega^2=c_\omega^2=d_\omega^2=1,\;\;b_\omega c_\omega d_\omega=1.
\end{equation}
We call the above relations the \emph{standard relations}.

Given a finitely generated group $G$ with a finite generating set $S$ let $\Gamma=\Gamma(G,S)$ be its Cayley graph and $B_n=B_n(G,S)$ be the ball of radius $n$ around the identity element of $G$ in $\Gamma$. The growth function of $G$ is defined as \begin{equation}\gamma(n)=\gamma_{G,S}(n)=|B_n|,n\in\mathbb N.\end{equation} The set of growth functions of groups is equipped with the (non-total) order $\preccurlyeq$. One says that $\gamma_1\preccurlyeq\gamma_2$ if there exists a number $C\in\mathbb N$ such that $\gamma_1(n)\leqslant \gamma_2(Cn)$ for all $n\in\mathbb N$. If $\gamma_1\preccurlyeq\gamma_2$ and $\gamma_2\preccurlyeq \gamma_1$ the growth functions $\gamma_1$ and $\gamma_2$ are called equivalent and we write $\gamma_1\thicksim\gamma_2$. If $\liminf\limits_{n\to\infty}\sqrt[n]{\gamma(n)}>1$, the group $G$ is said to have an \emph{exponential growth}. If $\gamma(n)\leqslant Cn^k$ for all $n\in\mathbb N$, where $C,k>0$, the group $G$ is said to have a \emph{polynomial growth}. Finally, the group $G$ is said to have an \emph{intermediate growth} if it has neither an exponential nor a polynomial growth.

Recall that $\Omega_2$ is the subset of $\Omega$ consisting of sequences with at lease two symbols from $\{0,1,2\}$ occurring infinitely many times. In \cite{Grigorchuk-DegreesGrowth-84}
the second author showed
\begin{thm}\label{ThGrowth} For $\omega\in \Omega_2$ the groups $G_\omega$ have intermediate growth. Moreover, there is a continuum $\widetilde\Omega\subset\Omega_2$ such that for $\omega\in\widetilde\Omega_2$ the groups $G_\omega$ have pairwise non-equivalent growth functions.
\end{thm}
\begin{figure}\centering\epsfig{file=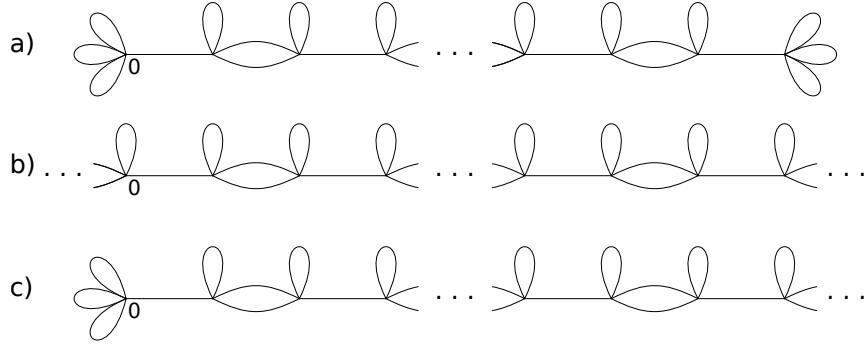,width=0.95\linewidth}
\caption{The graphs $a)$ $\Upsilon_n$, $b)$ $\Upsilon_\infty$ and $c)$ $\Upsilon_\infty^0$.} \end{figure}
Let us briefly recall the notion of the Schreier graph of an action of a group on a set. If a group $G$ with a generating set $S$ acts transitively on a set $U$, the Schreier graph $\Gamma$ of this action has the vertex set isomorphic to $U$. Two vertices $u_1,u_2\in U$ are connected by a directed edge labeled by a generator $s\in S$ if and only if $u_2=su_1$. Sometimes it is convenient to forget about the labels on the edges and the directions of the edges and consider unlabeled non-directed Schreier graphs.

Notice that automorphisms of the binary rooted tree $T$ preserve the levels $V_n$, $n\in\mathbb Z_+$, of the vertex set $V$. By definition, the boundary of $T$ is the set $\partial T$ of infinite paths of $T$ which start at $V_0$ and visit each of the levels $V_n$, $n\in\mathbb N$, exactly once. As a set, $\partial T$ is canonically isomorphic to $\{0,1\}^\mathbb N$. Automorphisms of $T$ act on $\partial T$ by permuting the infinite paths. The Schreier graphs for the actions of the group $\mathcal G$ on the sets $V_n$ and on $\partial T$ were described in  \cite{BartholdiGrigorchuk-Spectrum-00}.

For $n\in \mathbb N$ denote by $\Gamma_n$ the (unlabeled non-directed) Schreier graph of the action of $G_\omega$ on $V_n$ with respect to the generating set $S_\omega=\{a,b_\omega,c_\omega,d_\omega\}$. For a point $x\in\partial T$ denote by $\Gamma_x$ the (unlabeled non-directed) Schreier graph of the action of $G_\omega$ on the orbit of $x$ with respect to $S_\omega$. Introduce the graphs $\Upsilon_n$ with $2^n$ vertices identified with $\{0,1,\ldots,2^n-1\}$ such that
\begin{itemize}
\item[$1.$] $0$ and $2^n-1$ each has three loops attached;
\item[$2.$] each other vertex has one loop attached;
\item[$3.$] each even $i$ is connected by one edge to $i+1$;
\item[$4.$] each odd $i$ (except $2^{n-1}-1$) is connected by two edges to $i+1$.
\end{itemize}
Let $\Upsilon_\infty$ be the graph with the vertex set identified with $\mathbb Z$ and the edge set satisfying
$$ 2'. \;\;\text{each vertex has one loop attached}$$ and the conditions $3$ and $4.$ Finally, denote by $\Upsilon_\infty^0$ the graph with the vertex set identified with $\mathbb N\cup\{0\}$ and the edge set satisfying
$$1'.\;\;0\;\;\text{has three loops attached}$$ and the conditions $2-4.$ Denote by $\mathrm{r}=1^\infty$ the right most point of $\partial T$.
\begin{lem}\label{LmSchreierGamman} Let  $\omega\in\Omega$. For any $n\in\mathbb N$ the Schreier graph without labels $\Gamma_n$ is isomorphic to $\Upsilon_n$. For any $x\in G_\omega\mathrm{r}\subset \partial T$ the Schreier graph without labels $\Gamma_x$ is isomorphic to  $\Upsilon_\infty^0$. For any $x\in\partial T\setminus G_\omega\mathrm{r}$ the Schreier graph $\Gamma_x$ is isomorphic to  $\Upsilon_\infty$.
\end{lem}
\begin{proof} The case $n=1$ is trivial. Let $n>1$. By construction, the two right most vertices $u_1,u_2$ of $V_n$ satisfy:
$$a u_i\neq u_i,\;\;b_\omega u_i=c_\omega u_i=d_\omega u_i=u_i,\;\;i=1,2.$$ For any other vertex $v\in V_n$ one has $a v\neq v$, exactly one of the elements $b_\omega ,c_\omega ,d_\omega $ leaves $v$ fixed and the two other send $v$ to a vertex distinct from both $v$ and $a_\omega v$. Taking into account that $G_\omega $ acts transitively on $V_n$ we conclude from the above that $\Gamma_n$ is isomorphic to $\Upsilon_n$. The statements about $\Gamma_x$ follow from considering the limit points of the sequence of graphs $\{\Gamma_n\}_{n\in\mathbb N}$.
\end{proof}
\noindent In particular, the Schreier graphs $\Gamma_n$ and $\Gamma_x$ for any $\omega$ coincide with the corresponding Schreier graphs for the action of $\mathcal G$. Recall that $\partial T\cong\{0,1\}^\mathbb N$. Equip $\partial T$ with the Bernoulli measure $\mu=\{\tfrac{1}{2},\tfrac{1}{2}\}^\mathbb N$. From \cite{BartholdiGrigorchuk-Spectrum-00}, Theorem 3.6 and Corollary 4.4, we obtain the following:
\begin{prop}\label{PropSpecGamma} For $\mu$-almost every $x\in\partial T$ one has:
$$\overline{\bigcup\limits_{n\in\mathbb N}\spec(\Gamma_n)}= [-\tfrac{1}{2},0]\cup [\tfrac{1}{2},1]=\spec(\Gamma_x).$$
\end{prop}
\noindent Proposition \ref{PropSpecGamma} will be used to prove Theorem \ref{PropGrigSpec}.

\section{Graph coverings}\label{SecGraphs}
Let us recall some definitions and introduce some notations. Let $\Gamma=(V,E)$ be a non-directed graph where $V$ is the vertex set and $E$ is the edge set of $\Gamma$. In this paper we consider only graphs  of uniformly bounded degree. We allow $\Gamma$ to have multiple edges and loops. In this note we assume all graphs to have uniformly bounded degree. For a vertex $v$ of $\Gamma$ let $E_v$ stand for the set of edges adjacent to $v$. For $e\in E_v$ we consider $v$ as the first vertex of $e$ and denote by $r_v(e)$ be the second vertex of $e$, \ie $r_v(e)=v$ if $e$ is a loop and $r_v(e)$ is the vertex of $e$ not equal to $v$ otherwise. The Markov operator on the space $l^2(V)$ (which we also denote by $l^2(\Gamma)$) is the bounded linear operator $M=M_\Gamma$ given by $$(Mf)(v)=\frac{1}{d(v)}\sum\limits_{e\in E_v}f(r_v(e)),\;\;\text{for}\;\; f\in l^2(\Gamma),v\in V,$$ where $d(v)$ is the degree of the vertex $v$. Here we assume that each loop at a vertex $v$ contributes one edge to $E_v$ and contributes $1$ to the degree of $v$.

 More generally, assume that there is a weight $\alpha_{v,e}\in\mathbb C$ associated to every pair $(e,v)\in E\times V$ where $v$ is adjacent to $e$. We call $\Gamma$ a weighted graph in this case. We say that a weighted graph $\Gamma$ is uniformly bounded if it has a uniformly bounded degree and the set of weights $\alpha=\{\alpha_{v,e}\}$ is uniformly bounded. For a uniformly bounded weighted graph $\Gamma$ one can introduce a bounded Laplace type operator $H_\alpha$ on $l^2(\Gamma)$ by
$$(H_\alpha f)(v)=\sum\limits_{e\in E_v}\alpha_{v,e}f(r_v(e)),\;\;\text{for}\;\; f\in l^2(\Gamma).$$ The operator $H_\alpha$ is a generalization of the discrete Laplacian on $\Gamma$ or of the Markov operator on $\Gamma$. This bounded linear operator is self-adjoint whenever $\alpha_{v,e}=\overline{\alpha_{w,e}}$ for every triple $e,v,w$ where $e$ is an edge connecting vertices $v$ and $w$. It coincides with the Markov operator $M$ if for any pair $v,e$ of a vertex and an edge adjacent to it one has $\alpha_{v,e}=\frac{1}{\deg(v)}$.

For any operator $A$ define by $\sigma(A)$ the spectrum of $A$.
 We notice that when checking if $\sigma(H_\alpha)$ contains a particular value $\lambda\in\mathbb C$ we can restrict our attention to the case of a self-adjoint (and even a positive) operator $H_\alpha$.
 \begin{lem}\label{LmEquivNorm} Let $A$ be any bounded nonzero linear operator on a Hilbert space and let $R\geqslant 2\|A\|$. Then
$$\lambda\in \sigma(A)\;\;\Leftrightarrow\;\;1\in \sigma(\mathrm{I}-
\tfrac{1}{R^2}(A-\lambda\mathrm{I})(A-\lambda\mathrm{I})^{*}),$$
where $\mathrm{I}$ is the identity operator.
\end{lem}\noindent The proof is straightforward and we leave it as an exercise to the reader. Using Lemma \ref{LmEquivNorm} we obtain that for any uniformly bounded weighted graph $\Gamma$ with the set of weights $\alpha$, any $\lambda\in\mathbb C$ and any sufficiently large $R$ one has:
$$\lambda\in \sigma(H_\alpha)\;\;\Leftrightarrow\;\;1\in \sigma(\mathrm{I}-
\tfrac{1}{R^2}(H_\alpha-\lambda\mathrm{I})
(H_\alpha-\lambda\mathrm{I})^{*}).$$ The operator $\mathrm{I}-
\tfrac{1}{R^2}(H_\alpha-\lambda\mathrm{I})
(H_\alpha-\lambda\mathrm{I})^{*}$ is a positive operator of the form $H_\beta$ for some uniformly bounded set of weights on a new graph $\Gamma'=(V',E')$ with $V'=V$ and $E'$ consisting of all possible pairs of adjacent edges from $E$.

Recall that a graph $\widetilde\Gamma=(\widetilde E,\widetilde V)$ is a covering graph of a graph $\Gamma=(E,V)$ (synonymously, $\widetilde \Gamma$ covers $\Gamma$) if there exist surjective maps $\phi_1:\widetilde V\to V$ and $\phi_2:\widetilde E\to E$  which form a local isomorphism, \ie $\phi_2$ maps bijectively $\widetilde E_{\tilde v}$ onto $E_{\phi_1(\tilde v)}$ for each $v\in \widetilde V$. To simplify the notations we will use one letter $\phi$ for both of the maps $\phi_1$ and $\phi_2$. The map $\phi$ is called a covering map. It is straightforward to check that covering graphs have the following lifting property. If $\widetilde\Gamma$ is a covering graph for $\Gamma$ via a covering map $\phi$ then for any vertex $v$ in $\Gamma$, any path $\gamma$ in $\Gamma$ starting at $v$ and any $\tilde v\in\phi^{-1}(v)$ there exists a unique lift of the path $\gamma$ to a path starting at $\tilde v$.

Let $\phi$ be a covering map from $\widetilde\Gamma$ to $\Gamma$. Given systems of weights $\alpha=\{\alpha_{v,e}\},\tilde\alpha=\{\tilde\alpha_{\tilde v,\tilde e}\}$ associated to graphs $\Gamma$ and $\widetilde \Gamma$ correspondingly we say that the weighted graph $\widetilde\Gamma$ covers the weighted graph $\Gamma$ if $\tilde\alpha_{\tilde e,\tilde v}=\alpha_{\phi(\tilde e),\phi(\tilde v)}$ for any edge $\tilde e$ of $\widetilde \Gamma$ and any vertex $\tilde v$ of $\tilde e$. In this section we investigate how the spectra of the operators of the form $H_\alpha$ change under coverings.

For a vertex $v$ of a graph $\Gamma=(V,E)$ and a number $n\in\mathbb N\cup\{0\}$ denote by $B_n(v)$ the ball of radius $n$ around $v$. For a set of vertices $A\subset V$ let $$B_n(A)=\bigcup\limits_{v\in A}B_n(v).$$ A graph has subexponential growth if there exists a vertex $v\in V$ such that \begin{equation}\label{EqSubexp}\liminf\limits_{n\to\infty}(|B_n(v)|)^{\frac{1}{n}}=1,\end{equation} where $|\cdot|$ stands for the number of elements of a set. Notice that for a graph of uniformly bounded degree the condition \eqref{EqSubexp} is equivalent for all vertices $v$. The definition of an amenable graph was given in \cite{CeccheriniGrigorchukHarpe-Amenability-99} where various equivalent conditions of amenability of graphs of uniformly bounded degree were established. Recall that a graph $\Gamma$ of uniformly bounded degree is amenable if there exists an increasing sequence $(F_k)_{k\in\mathbb N}$ of finite sets of vertices (called F{\o}lner sequence) such that
$$\bigcup\limits_{k\in\mathbb N}F_k=V\;\;\text{and}\;\;\lim\limits_{k\to\infty}\frac{| B_1(F_k)\setminus F_k|}{|F_k|}=0.$$ Every graph of subexponential growth is amenable since it has a F{\o}lner sequence of the form $F_k=B_{n_k}(v)$, where the sequence $n_k$ is such that $\frac{|B_{n_k(v)}|}{|B_{n_k+1}(v)|}\to 1$ when $k\to\infty$. $\Gamma$ is amenable if and only if for the Markov operator $M$ associated to $\Gamma$ one has $\|M\|=r(M)=1$, where $r(M)$ is the spectral radius of $M$ (see \eg \cite{CeccheriniGrigorchukHarpe-Amenability-99}).

The proof of Theorem \ref{ThGraphCov} is inspired by the proof of Proposition 3.9 from \cite{BartholdiGrigorchuk-Spectrum-00}.
\begin{proof}[Proof of Theorem \ref{ThGraphCov}] First, we notice that by Lemma \ref{LmEquivNorm} without loss of generality we may assume that $H_1$ and $H_2$ are positive semi-definite. In addition, multiplying all the weights of $\Gamma_1$ and $\Gamma_2$ by the same small number $c>0$ we may assume that all weights are smaller than one by absolute value. Let $\phi:\Gamma_1\to\Gamma_2$ be the covering. Denote by $\|\cdot\|$ the $l^2$-norm on either $l^2(\Gamma_1)$ or $l^2(\Gamma_2)$.

 $a)$ Assume that $\Gamma_1$ is amenable and $\Gamma_2$ is finite. Let $\lambda\in\sigma(H_2)$. Then there exists $f\in l^2(\Gamma_2)$ of norm $1$ such that $H_2f=\lambda f$. Fix a vertex $w\in \Gamma_2$ such that $f(w)\neq 0$. Let $\{F_k\}$ be a F{\o}lner sequence for $\Gamma_1$. Let $N$ be the number of vertices of $\Gamma_2$.
 For $k\in\mathbb N$ introduce
\begin{align*}f_k\in l^2(\Gamma_1),\;f_k(x)=\left\{\begin{array}{ll}f(\phi(x)),&\text{if}\;x\in B_{N+1}(F_k),\\
0,&\;\text{otherwise};\end{array}\right. \\ A_{k,r}=\{x\in  B_r(F_k):\phi(x)=w\},\;\;\alpha_{k,r}=|A_{k,r}|.\end{align*}
Observe that for any $x\in B_N(F_k)$ one has $$(H_1f_k)(x)-\lambda f_k(x)=(H_2f)(\phi(x))-\lambda f(\phi(x))=0.$$
Therefore,
$$\|H_1f_k-\lambda f_k\|^2=\sum\limits_{x\in B_{N+2}(F_k)\setminus B_N(F_k)}|(H_1f_k)(x)-\lambda f_k(x)|^2.$$ For each $y\in \Gamma_2$ there exists a path $\gamma_y$ in $\Gamma_2$ of length at most $N$ joining $y$ and $w$. For every $y\in \Gamma_2$ and $x\in\phi^{-1}(y)$ there exists a unique lift of $\gamma_y$ to a path in $\Gamma_1$ starting at $x$. Denote by $w_x$ the other end-point of this lift. If $x\in F_k$ then $w_x$ belongs to $A_{k,N}$. Moreover, for any distinct vertices $t,x\in\phi^{-1}(y)\cap F_k$ one has $w_x\neq w_t$. We obtain that for every $y\in \Gamma_2$ one has:
$$|\phi^{-1}(y)\cap F_k|\leqslant \alpha_{k,N}.$$ Using the above inequality for all $y\in \Gamma_2$ we get:
$|F_k|\leqslant N\alpha_{k,N}.$
It follows that $$\|f_k\|^2\geqslant \alpha_{k,N}|f(w)|^2\geqslant\frac{1}{N}|f(w)|^2|F_k|.$$
Setting $\tilde f_k=\frac{f_k}{\|f_k\|}$ we arrive at:
$$\|H_1\tilde f_k-\lambda\tilde f_k\|^2\leqslant \frac{C| B_{N+2}(F_k)\setminus
B_N(F_k)|}{|F_k|},$$ where $C$ depends only on $\Gamma_1,\Gamma_2$ and $f$. When $k\to\infty$ the latter converges to $0$ since $F_k$ is a F{\o}lner sequence. It follows that $\lambda\in\sigma(H_1)$ which finishes the proof of part $a)$.

$b)$ Let $\Gamma_1$ have subexponential growth. Fix $\epsilon>0$  and $\lambda\in \sigma(H_2)$. Since $H_2$ is self-adjoint there exists $f\in l^2(\Gamma_2)$ of $l^2$-norm 1 such that
$\|H_2f-\lambda f\|<\epsilon$. Without loss of generality we may assume that the set  $S=\supp (f)$ of vertices at which $f$ is nonzero is finite. Fix any vertex $v\in\Gamma_1$ and $N$ such that $S\subset B_N(\phi(v))$.

For an integer $k\geqslant 0$ introduce
\begin{align*}f_k\in l^2(\Gamma_1),\;f_k(x)=\left\{\begin{array}{ll}f(\phi(x)),
&\text{if}\;x\in B_k(v),\\
0,&\;\text{otherwise};\end{array}\right.
 \\ A_k=\{x\in  B_k(v):\phi(x)=\phi(v)\},\;\;\alpha_k=|A_k|.\end{align*} We also set $A_k=\varnothing$ and $\alpha_k=0$ for $k<0$.

Observe that for every $x\in B_{k-1}(v),k\geqslant 1,$ one has
 $$|(H_1f_k)(x)-\lambda f_k(x)|=|(H_2f_k)(\phi(x))-\lambda f(\phi(x))|.$$
For every $y\in S$ fix a path $\gamma_y\subset B_N(\phi(v))$ of length at most $N$ joining $y$ and $\phi(v)$. By the lifting property for every $y\in S$, $l\in\mathbb N$ and $x\in\phi^{-1}(y)\cap B_{l}(v)$ there exists a unique lift of $\gamma_y$ to a path in $\Gamma_1$ starting at $x$. Denote the other end-point of this lift by $v_x$. Observe that $$v_x\in A_{d+N}\setminus A_{d-N}\subset A_{l+N},$$ where $d=d(x,v_x)$ is the geodesic distance between $x$ and $v_x$ in $\Gamma_1$. Moreover, for a fixed $y\in S$ and distinct vertices $x,t\in\phi^{-1}(y)\cap B_{l}(v)$ one has $v_x\neq v_t$. Notice also that for every $w\in \phi^{-1}(\phi(v))\cap B_l(v)$ and every $y\in S$ there exists a unique $t\in \phi^{-1}(y)\cap B_{l+N}(v)$ such that $v_t=w$. We obtain that for every $y\in S$ and every $k>N$ one has:
\begin{align*}|\phi^{-1}(y)\cap B_{k-N}(v)|\leqslant \alpha_{k}\;\;\text{and}\\ |\phi^{-1}(y)\cap B_k(v)\setminus B_{k-N}(v)|\leqslant \alpha_{k+N}-\alpha_{k-2N}.\end{align*}
Therefore,
\begin{align*}\|H_1f_k-\lambda f_k\|^2=\sum\limits_{x\in B_{k-N}(v)}|(H_1f_k)(x)-\lambda f_k(x)|^2+\\
 \sum\limits_{x\notin B_{k-N}(v)}|(H_1f_k)(x)-\lambda f_k(x)|^2 \leqslant \alpha_k\|H_2f-\lambda f\|^2+2|S|(\alpha_{k+N}-\alpha_{k-2N}).\end{align*}
On the other hand, by a similar argument, $\|f_k\|^2\geqslant \alpha_{k-N}\|f\|^2=\alpha_{k-N}$. Setting $\tilde{f}_k=f_k/\|f_k\|$ we arrive at:
$$\|H_1\tilde f_k-\lambda \tilde f_k\|^2\leqslant \epsilon^2\frac{\alpha_k}{\alpha_{k-N}}+
|S|\frac{\alpha_{k+N}-\alpha_{k-2N}}{\alpha_{k-N}}.$$
If $\liminf\limits_{k\to\infty}\frac{\alpha_{k+1}}{\alpha_k}>1$ then there exists $C>0$ and $\mu>1$ such that $\alpha_k\geqslant C\mu^k$ (and therefore $|B_k(v)|\geqslant C\mu^k$) for any $k>0$. Therefore, $\Gamma_1$ has exponential growth. Thus, $\lim\limits_{k\to\infty}\frac{\alpha_{k+1}}{\alpha_k}=1$. Taking into account that $\alpha_l$ is non-decreasing in $l$ we obtain that there exists $k$ such that $\|H_1\tilde f_k-\lambda \tilde f_k\|<2\epsilon$. This shows that $\lambda\in \sigma(H_1)$ and finishes the proof of part $b)$.
\end{proof}

\section{Proof of Theorem \ref{PropGrigSpec}}
Given a unitary representation $\pi$ of a group $G$ we extend it to a representation of $\mathbb C[G]$ by linearity. Recall that a unitary representation $\rho$ is weakly contained in a unitary representation $\eta$ of a discrete group $G$ if and only if $$\sigma(\rho(m))\subset\sigma(\eta(m))\;\;\text{for every}\;\;m\in\mathbb C[G].$$

\begin{proof}[Proof of Theorem \ref{PropGrigSpec}] Consider $\tilde t=\frac{\tilde b+\tilde c+\tilde d-1}{2}\in\mathbb C[G]$. One can easily verify that $(\tilde t)^2=1$. Let $\lambda_G$ be the regular representation of $G$. Then $\lambda_G(\tilde t)$ is a square root of the identity and self-adjoint, therefore it is a unitary operator. Let $D_\infty=<s,t:s^2=t^2=1>$ be the infinite dihedral group. We obtain that the assignment $\pi(s)=\lambda_G(\tilde a),\pi(t)=\lambda_G(\tilde t)$ extends to a unitary representation $\pi$ of $D_\infty$. By Hulanicki Theorem \ref{ThHulanicki}, $\pi\prec\lambda_{D_\infty}$. Therefore,
\begin{align*}\sigma(\lambda_G(\tfrac{1}{4}(\tilde a+\tilde b+\tilde c+\tilde d)))=
\sigma(\pi(\tfrac{1}{4}\tilde a+\tfrac{1}{2}\tilde t+\tfrac{1}{4}))\subset
\sigma(\lambda_{D_\infty}(\tfrac{1}{4}s+\tfrac{1}{2} t+\tfrac{1}{4}))
.\end{align*}
Notice that $\|\lambda_{D_\infty}(\tfrac{1}{4}s+\tfrac{1}{2} t)\|\leqslant\tfrac{1}{4}\|\lambda_{D_\infty}(s)\|+
\tfrac{1}{2}\|\lambda_{D_\infty}(t)\|= \tfrac{3}{4}$, therefore $\sigma(\lambda_{D_\infty}(\tfrac{1}{4}s+\tfrac{1}{2} t))\subset [-\tfrac{3}{4},\tfrac{3}{4}]$. Observe also that for any $\lambda\in (-\tfrac{1}{4},\tfrac{1}{4})$ we have that
$$\lambda_{D_\infty}(\tfrac{1}{4}s+\tfrac{1}{2} t-\lambda)=\lambda_{D_\infty}(t)\lambda_{D_\infty}(\tfrac{1}{4}ts-\lambda t+\tfrac{1}{2})$$ is invertible, since $\|\lambda_{D_\infty}(\tfrac{1}{4}ts-\lambda t)\|\leqslant \tfrac{1}{4}\|\lambda_{D_\infty}(ts)\|+
|\lambda|\|\lambda_{D_\infty}(t)\|<\tfrac{1}{2}$. We obtain that
$$\sigma(\lambda_{D_\infty}(\tfrac{1}{4}s+\tfrac{1}{2} t))\subset [-\tfrac{3}{4},-\tfrac{1}{4}]\cup [\tfrac{1}{4},\tfrac{3}{4}]\;\Rightarrow\;\sigma(\lambda_{D_\infty}(\tfrac{1}{4}s+\tfrac{1}{2} t+\tfrac{1}{4}))\subset [-\tfrac{1}{2},0]\cup [\tfrac{1}{2},1].$$

On the other hand, the surjection $\varphi$ induces a graph covering $\Gamma(G,\widetilde S)\to \Gamma( G_\omega,S_\omega)$, where $\widetilde S=<\tilde a,\tilde b,\tilde c,\tilde d>$ and $S_\omega=<a,b_\omega,c_\omega,d_\omega>$. Let $\Gamma_n$ be the Schreier graph of the action of $ G_\omega$ on the $n$th level of the binary regular rooted tree. Then $\Gamma( G_\omega,S_\omega)$ covers $\Gamma_n$. From Theorem \ref{ThGraphCov} we obtain that $\spec(\Gamma_n)\subset \spec(\Gamma(G,\widetilde S))$ for every $n$. By Proposition \ref{PropSpecGamma},  we have
$$\overline{\bigcup\limits_{n\in\mathbb N}\spec(\Gamma_n)}= [-\tfrac{1}{2},0]\cup [\tfrac{1}{2},1].$$ Therefore, $\spec(G)=\spec(\Gamma(G,\widetilde S))\supset [-\tfrac{1}{2},0]\cup [\tfrac{1}{2},1]$. This finishes the proof.
\end{proof}
\section{Proof of Theorem \ref{ThMain}}
To prove Theorem \ref{ThMain}, for each $\omega\in\Omega_2$ we construct a continuum of pairwise non-isomorphic groups covering $G_\omega$ and satisfying the conditions of Theorem \ref{PropGrigSpec}.

Introduce the group $$\Lambda=<\hat a,\hat b,\hat c,\hat d:\hat a^2=\hat b^2=\hat c^2=\hat d^2=\hat b\hat c\hat d=1>.$$ Observe that $\Lambda\simeq\mathbb Z_2*(\mathbb Z_2\times\mathbb Z_2)$, where $*$ denotes the free product. For $n\in\mathbb N\cup\{\infty\}$, where $\infty$ denotes the cardinality of $\mathbb N$, we denote by $\mathbb F_n$ the free group with $n$ generators. An easy application of the Reidemeister-Schreier method implies:
\begin{lem}\label{LmLambda'} The commutator subgroup $\Lambda'=[\Lambda,\Lambda]$ is freely generated by $[\hat a,\hat b],[\hat a,\hat c],[\hat a,\hat d]$, and so is isomorphic to $\mathbb F_3$.
\end{lem}
\noindent
The map $\psi(\hat a)=a,\psi(\hat b)=b,\psi(\hat c)=c,\psi(\hat d)=d$ extends to a homomorphism $\psi$ from $\Lambda$ onto $ G_\omega$. Let $\Delta_\omega$ be the kernel of $\psi$. To study $\Delta_\omega$, let us recall the presentation by generators and relations for $G_\omega$ constructed in \cite{Muntyan-Dissertation-09}.

An infinite word $\omega\in \Omega$ is called \emph{almost constant} if for some $k\in\{0,1,2\}$ one has $\omega_i=k$ for all but finitely many $i\in\mathbb N$.  For every non-constant word $\omega\in \Omega$ there exists a unique permutation $(x,y,z)=(x_\omega,y_\omega,z_\omega)$ of the set $\{0,1,2\}$ and $n>2$ such that $\omega$ is of one of the following three forms:
\begin{equation}\label{EqOmegaTypes}1)\; \omega=\underbrace{xx\ldots xy}_n\ldots,\;\;2)\; \omega=\underbrace{xy\ldots yx}_n\ldots,\;\;3)\;\omega= \underbrace{xy\ldots yz}_n\ldots.
 \end{equation} We will say that $\omega$ is of the type $1)$, $2)$ or $3)$ respectively. Identify $0$ with $d$, $1$ with $c$ and $2$ with $b$. For an  alphabet $\mathcal A$ let $\mathcal A^*$ be the collection of all words of finite length over $\mathcal A$. Following Proposition II.22 from \cite{Muntyan-Dissertation-09} introduce collections of words $U_1^\omega\subset \{a,b_\omega,c_\omega,d_\omega\}^*$ as follows. Correspondingly to the type of $\omega$ we set
\begin{itemize} \item[$1)$] $U_1^\omega=\{(xaya)^4,(xa(ya)^{2k})^4,k=1,\ldots,2^{n-1}\}$,
\item[$2)$] $U_1^\omega=\{(xayaya)^4,(xaya)^{2^n}\}$,
\item[$3)$] $U_1^\omega=\{(xayaya)^4,(zaya)^{2^n}\}$.
\end{itemize} Notice that the above formulas define $U_1^\omega$ uniquely for every non-constant $\omega\in\Omega$. Denote by $U_0^\omega$ the set of standard relations in $G_\omega$: $$U_0^\omega=\{a^2,b^2,c^2,d^2,bcd\}.$$

For $\omega\in\Omega$ let $\omega'$ be the word obtained from $\omega$ by removing the first letter.  Introduce a substitution $\phi_\omega: \{a,b_{\omega'},c_{\omega'},d_{\omega'}\}^*\to \{a,b_\omega,c_\omega,d_\omega\}^*$ by:
$$\phi_\omega(x_{\omega'})=x_\omega,\;\phi_\omega(y_{\omega'})=y_\omega,\;\phi_\omega(z_{\omega'})=z_\omega,
\;\phi_\omega(a)=a y_{\omega}a.$$
Denote recursively $U_{k+1}^\omega=\phi_\omega(U_k^{\omega'})$ for $k\geqslant 1$.
In \cite{Muntyan-Dissertation-09}, Proposition II.24, Muntyan showed the following:
\begin{thm}\label{ThMuntyan} Let $\omega\in\Omega$ be non-almost-constant. Then $$G_\omega\cong \left\langle a,b,c,d:\bigcup\limits_{k\in\mathbb N\cup \{0\}} U_k^\omega\right\rangle.$$
\end{thm}
\noindent Observe that when $\omega=012012\ldots$ the presentation for $G_\omega$ from Theorem \ref{ThMuntyan} coincides with the presentation \eqref{EqGrigPres}.
\begin{lem}\label{LmOmega} $\Delta_\omega$ is a normal subgroup of the commutator subgroup $\Lambda'$ and is isomorphic to a free group of infinite rank $\mathbb F_\infty$.
\end{lem}
\begin{proof} A word $\hat w\in \{\hat a,\hat b,\hat c,\hat d\}^*$ defines an element of $\Delta_\omega$ if and only if the word $w$ obtained by removing all hats from $\hat w$ defines the trivial element of $G_\omega$. Therefore, to show that $\Delta_\omega<\Lambda'$ it is sufficient to check that the following condition is satisfied:
\begin{align}\label{EqUkomegaCond}\begin{split}\text{for any}\;\;k\in\mathbb N\;\;\text{and any word}\;\;u\in U_k^\omega\\
 \text{the word}\;\;\hat u\in \{\hat a,\hat b,\hat c,\hat d\}^*\;\;\text{belongs to}\;\;\Lambda'.\end{split}\end{align} Here $\hat u$ stands for the word obtained from $u$ by adding hats to letters of $u$. For $k=1$ the condition \eqref{EqUkomegaCond} can be checked by direct computations.  For instance, for any permutation $(x,y,z)$ of $(b,c,d)$ for the first word of $U^1_\omega$ with $\omega$ of type $1)$ we have that $$(\hat x\hat a\hat y\hat a)^4=([\hat x,\hat a][\hat a,\hat z][\hat y,\hat a])^4 \in\Lambda'.$$
 Since the commutator subgroup $\Lambda'$ is a fully characteristic subgroup of $\Lambda$, it is invariant under the endomorphisms generated by substitutions.
 It follows by induction that the condition \ref{EqUkomegaCond} is satisfied for every $k\in\mathbb N$ and so $\Delta_\omega<\Lambda'$. As a kernel of a homomorphism, $\Delta_\omega$ is a normal subgroup.
 Since it is of infinite index, $\Delta_\omega$ is isomorphic to $\mathbb F_\infty$ (see \cite{MagnusKarrasSolitar-CombinatorialGroup-66}, Theorem 2.10).
\end{proof}
Let us recall the notion of a \emph{verbal subgroup}. For details we refer the reader \eg to \cite{Neumann-VariatiesGroups-67}. Let $\mathcal W\subset \mathcal A^*$ be a set of finite words over an alphabet $\mathcal A$ and $G$ be any group. For any map $\phi:\mathcal A\to G$ and any word $w=x_1\ldots x_k\in\mathcal W$ denote by $g_{\phi,w}$ the element of $G$ obtained by replacing each letter of $w$ by its image under $\phi$:
$$g_{\phi,w}=\phi(x_1)\cdots\phi(x_k).$$ The verbal subgroup of $G$ defined by $\mathcal W$ is the group generated by all words of the form $g_{\phi,w}$. Let us state the following result of Ol'$\check{\text{s}}$anski$\breve{\i}$ \cite{Olshanski-CharacteristicSubgroups-74}:
\begin{thm}\label{ThOls} The group $F_\infty$ has a continuum of distinct verbal subgroups $V_i, i\in I$, such that the quotient $F_\infty/V_i$ is locally finite and solvable for every $i\in I$.
\end{thm}
\noindent Let us identify $F_\infty$ with $\Delta_\omega$ and view $V_i,$ $i\in I$, from Theorem \ref{ThOls} as subgroups of $\Delta_\omega$. For every $i$, since $V_i$ is a verbal subgroup of $\Delta_\omega$ and $\Delta_\omega$ is a normal subgroup of $\Lambda$ we obtain that $V_i$ is a normal subgroup of $\Lambda$. Since $\Lambda/\Delta_\omega\cong G_\omega$ and $\Delta_\omega/V_i$ are amenable, $\Lambda/V_i$ is also amenable for every $i\in I$. Notice that each of the groups $\Lambda/V_i$ satisfies the conditions of Theorem \ref{PropGrigSpec}. We obtain that their spectra all coincide with $[-\tfrac{1}{2},0]\cup[\tfrac{1}{2},1]$. Since $\Lambda$ is finitely generated, among the groups $\Lambda/V_i$ there is a continuum of pairwise non-isomorphic groups. This finishes the proof of Theorem \ref{ThMain}.

\section{Proof of Theorem \ref{ThSpec[-1,1]}}
Using the construction of P. Hall from \cite{Hall-Finiteness-54}, proofs of Theorems 7 and 8, we obtain:
\begin{prop}\label{ThContinuum2gen} There is a continuum of $2$-generated torsion-free step-3 solvable groups with all relations of even length.
\end{prop}
\begin{proof} Introduce the group $$A=\langle x_i,i\in\mathbb Z:[y_1,[y_2,y_3]]=1,y_1,y_2,y_2\in \{x_i,x_i^{-1}:i\in\mathbb Z\}^*\rangle,$$ where as before $\mathcal A^*$ denotes the set of all finite words over the alphabet $\mathcal A$.  The group $A$ is a free nilpotent group of nilpotency class 2 and of infinite rank. Let $\varphi:A\to A$ be the automorphism of $A$ induced by the shift map on the set of generators: $\varphi(x_i)=x_{i+1},i\in\mathbb Z$. Define the semi-direct product $B=A\rtimes_\varphi \mathbb Z$. Notice that $B$ is generated by $x_0$ and $\varphi$. Consider the subgroup $C$ of $B$ generated by the elements $$[x_i,x_j][x_k,x_l]^{-1},\;\;i,j,k,l\in\mathbb Z,\;\;i-j=k-l.$$ Clearly, $C$ is a normal subgroup of $B$. Let $K=B/C$. Then the center $Z(K)$ of $K$ is an abelian group of infinite rank generated by the commutators $[x_0,x_i],i\in\mathbb Z$. Therefore, $Z(K)$ has a continuum of subgroups.

Further, each subgroup $H<Z(K)$ is normal in $K$. Moreover the quotient $K/H$ is torsion free and solvable. Represent $K/H$ as $F_2/N$ such that the generators $\bar{x}_0$ and $\bar\varphi$ of $F_2$ are mapped to the classes of $x_0$ and $\varphi$ in $K/H$, respectively. The relations in $F_2/N$ are all of even length since they are products of commutators. Since the group of automorphisms of a finitely generated group is at most countable, there exists a continuum of pairwise non-isomorphic groups of the form $K/H$, $H<Z(K)$. This finishes the proof of Proposition \ref{ThContinuum2gen}
\end{proof}

Recall that for a countable group $G$ the regular representation $\lambda_G$ is defined on the Hilbert space $l^2(G)$ by
$$(\lambda_G(g)f)(h)=f(g^{-1}h),\;\;f\in l^2(G),\;\;h,g\in G.$$ Further, we use Proposition 3.7 from \cite{BartholdiGrigorchuk-Spectrum-00}:
\begin{prop}\label{PropHigsonKasparov} Let $\Gamma$ be a torsion-free amenable group with finite generating set $S=S^{-1}$ such that there is a map $\phi:\Gamma\to \mathbb Z/2\mathbb Z$ with $\phi(S)=1$. Let $\lambda_\Gamma$ be the regular representation of $\Gamma$. Then
$$\sigma(\sum\limits_{s\in S}\lambda_\Gamma(s))=[-|S|,|S|].$$
\end{prop}
\noindent Let $G$ be a $2$-generated torsion free solvable group with all relations of even length. Since $G$ is countable and soluble, it is amenable. Thus, the group $G$ satisfies the conditions of Proposition \ref{PropHigsonKasparov}. Therefore, $\spec(G)=[-1,1]$. Applying Proposition \ref{ThContinuum2gen} we finish the proof of Theorem \ref{ThSpec[-1,1]}.

\section{Remarks and open questions}
Given a finitely generated group $G$  together with the spectrum of $G$ it is natural to consider spectral measures of the Markov operator $M$. Since $M$ is a self-adjoint operator on $l^2(G)$ it admits a spectral decomposition
$$M=\int\limits_{\sigma(M)}\dd E(\lambda),$$ where $E(\lambda)$ is a projection-valued measure. Introduce the spectral measure $\mu$ of $M$ by
$$\mu(A)=(E(A)\delta_e,\delta_e),$$ where $A\subset \mathbb R$ is an $E$-measurable set, $\delta_e\in l^2(G)$ is the delta-function of the unit element in $G$ and $(\cdot,\cdot)$ is the scalar product of $l^2(G)$. Then $\supp(\mu)=\sigma(M)=\spec(G)$ and $\mu$ may contain more information about $G$ than just the spectrum $\spec(G)$ itself (see \cite{Kesten-RandomWalk-59}). Therefore, a natural question is:
\begin{Quest}\label{QuestSpectralMeas} Is it possible to determine the Cayley graph of a finitely generated group $G$ up to isometry  from the spectral measure $\mu$ of the associated Markov operator?
\end{Quest}\noindent The above question is wide open. In particular, we do not know what would be the answer if we restrict our attention to the groups considered in the present paper.
\begin{Quest}\label{QuestMeasGrig}
 $a)$ Is  it  correct  that for any two groups  from  the  family  $G_\omega$,  $\omega\in\Omega$, with  non-isometric  Cayley  graphs   the  corresponding  spectral  measures  are distinct? $b)$ Same question for the groups constructed by P. Hall and used in the proof of Theorem \ref{ThSpec[-1,1]}.
\end{Quest}\noindent
Another natural question (already mentioned in the Introduction) is whether or not Theorem \ref{ThGraphCov} (Weak Hulanicki Theorem for graphs) can be proven in the full generality:
\begin{Quest}\label{QuestHulanicki} Let $\Gamma_1$ be an amenable uniformly bounded connected weighted graph which covers a weighted graph $\Gamma_2$. Let $H_1,H_2$ be the Laplace type operators associated with $\Gamma_1$ and $\Gamma_2$, respectively. Is it true that $\sigma(H_2)\subset \sigma(H_1)$?
\end{Quest}
\section{Acknowledgements}
We are grateful to the anonymous referee for careful reading of our paper and numerous useful comments and suggestions.

\bibliographystyle{ijmart}
\bibliography{Bibliography}
\end{document}